\documentclass[letterpaper]{article}

\usepackage[english]{babel}
\usepackage[utf8x]{inputenc}
\usepackage[T1]{fontenc}

\usepackage[margin=1in]{geometry}

\usepackage{amsmath,amssymb,amsthm,bbm,bm}
\theoremstyle{definition}
\newtheorem{Problem}{Problem}
\newtheorem{Lemma}{Lemma}
\newtheorem{Proposition}{Proposition}
\newtheorem{Theorem}{Theorem}
\newtheorem{Corollary}{Corollary}
\usepackage{graphicx}
\usepackage[colorinlistoftodos]{todonotes}
\usepackage[colorlinks=true, allcolors=blue]{hyperref}
\renewcommand{\Re}{\mathbb{R}}
\newcommand{\boldX}{\mathbf{X}}
\newtheorem*{remark}{Remark}

\title{Optimal Covariance Control for Stochastic Systems\\ Under Chance Constraints}

\author{Kazuhide Okamoto\thanks{K. Okamoto is a Ph.D. Candidate at the School of Aerospace Engineering, Georgia Institute of Technology, Atlanta, GA 30332-0150, USA. Email: kazuhide@gatech.edu} \qquad
Maxim Goldshtein\thanks{M. Goldshtein is a PhD Student at the School of Aerospace Engineering, Georgia Institute of Technology, Atlanta, GA 30332-0150, USA. Email:maxg@gatech.edu} \qquad
Panagiotis Tsiotras\thanks{P. Tsiotras is a Professor at the School of Aerospace Engineering and the Institute for Robotics \& Intelligent Machines, Georgia Institute of Technology, Atlanta, GA 30332-0150, USA. Email: tsiotras@gatech.edu}}

\begin{document}
\maketitle

\begin{abstract}
This work addresses the optimal covariance control problem for stochastic discrete-time linear time-varying systems subject to chance constraints. Covariance steering is a stochastic control problem to steer the system state Gaussian distribution to another Gaussian distribution while minimizing a cost function. To the best of our knowledge, covariance steering problems have never been discussed with probabilistic chance constraints although it is a natural extension. In this work, first we show that, unlike the case with no chance constraints, the covariance steering with chance constraints problem cannot decouple the mean and covariance steering sub-problems. Then we propose an approach to solve the covariance steering with chance constraints problem by converting it to a semidefinite programming problem. The proposed algorithm is verified using two simple numerical simulations. 
\end{abstract}

\section{Introduction}

In this work we address a problem of finite-horizon stochastic optimal control for a discrete-time linear time-varying stochastic system with a fully-observable state, a given Gaussian distribution of the initial state, and a state and input-independent white-noise Gaussian diffusion with given statistics. 
The control task is to steer the system state to the target Gaussian distribution, while minimizing a state and control expectation-dependent cost. 
In addition to the boundary condition, in the aim of adding robustness to the controller under stochastic uncertainty, we consider \emph{chance constraints}, which restricts the probability of violating the state constraints to be less than a pre-specified threshold.

Since the Gaussian distribution can be fully defined by its first two moments, this problem can be described as a finite-time optimal mean and covariance steering of a stochastic time-varying discrete linear system, with a boundary conditions in form of given initial and final mean and covariance, and with constraints on the trajectory in form of probability function. 

This chance-constrained optimal covariance control problem is relevant to a wide range of control and planning task, such as decentralized control of swarm robots~\cite{Shahrokhi2016}, closed-loop cooling~\cite{Vinante2008}, and others, in which the state is more naturally described by it's distribution, rather than a fixed set of values. In addition, this approach is readily-applicable to a stochastic MPC framework~\cite{Farina2013}.

\subsection*{Covariance steering problem}

The problem of controlling the state covariance of a linear system goes back to the late 80s. The so-called covariance steering (or ``Covariance Assignment'') problem was first introduced by Hotz and Skelton~\cite{hotz1985covariance}, where they computed the state feedback gains of a linear time-invariant system, 
such that the state covariance  converges to a pre-specified value.
Since then, many works have been devoted to this problem of infinite-horizon covariance assignment, 
both for continuous and discrete time systems~\cite{iwasaki1992quadratic,xu1992improved,grigoriadis1997minimum, Hotz1987,Collins1987}.
Recently, the finite-horizon covariance control problem has been investigated by a number of researchers~\cite{chen2016optimalI,chen2016optimalII,chen2018optimalIII,	Chen2016Thesis, ridderhof2018uncertainty}, relating to the problems of Shr{\"{o}}dinger bridges~\cite{Schrodinger1931} and the Optimal Mass Transfer~\cite{Kantorovich1942}. Others, including our previous work, showed that the finite covariance control problem solution can be seen as a LQG with a particular weights~\cite{halder2016finite, goldshtein2017finite}, which can be also formulated (and solved) as a LMI problem~\cite{bakolas2016optimalACC,bakolas2016optimalCDC,bakolas2018finite}.

\subsection*{Chance-constrained control}
A chance-constrained optimization has been extensively studied since 50's, with purpose of system design with guaranteed performance under uncertainty~\cite{Geletu2013}. A stochastic model-predictive control design with a chance-constraints has been solved using various techniques (see~\cite{Farina2016} for an extensive review).

\subsubsection*{Main contribution}

The current paper contributes to this line of work by adding chance state constraints to the underlying stochastic optimal covariance  steering problem. The covariance control problem is reformulated as an LMI, with a decision variable that is quadratic in the cost function. The chance constrains are presented as a deterministic constraint on the mean and the covariance, and added to the LMI formulation, allowing an efficient solution with any generic LMI solvers. 

To the best of the authors' knowledge, this work is the first that solves the covariance-steering problem with chance constraints.

\subsubsection*{Paper structure}

The remainder of this paper is organized as follows. 
Section~\ref{sec:ProblemStatement} formulates the problem we wish to address in this paper.
We consider a time-varying stochastic linear system and a general $L_2$-norm objective function. 
Before introducing the proposed approach to deal with the chance constraints, in Section~\ref{sec:NoChanceConstraint}, we describe the case without chance constraints, in which case the mean and covariance steering part can be nicely decoupled into two independent subproblems, namely, one for the mean and one for the second moment of the state.
In Section~\ref{sec:ProposedApproach} we provide the solution to the covariance steering problem with chance constraints and we show how the solution of the problem can be cast as a convex optimization problem.
In Section~\ref{sec:NumericalSimulation} we validate the effectiveness  of the proposed approach via numerical simulations.
Finally, Section~\ref{sec:Summary} briefly summarizes the contribution of this work and proposes some possible future research directions.

\section{Problem Statement}\label{sec:ProblemStatement}
 
In this section we formulate the finite-horizon optimal covariance control problem for a stochastic, linear, time-varying system in discrete-time subject to chance constraints. 
We then reformulate the problem using concatenated state, input, and noise vectors and transition matrices. 

\subsection{Problem Formulation}

We consider the following discrete-time stochastic linear system (possibly time-varying) with additive uncertainty, described by the equations
\begin{equation}  \label{eq:SystemDynamics}
x_{k+1} = A_kx_k + B_ku_k + D_kw_k, 
\end{equation}
where $k = 0,1,\ldots,N-1$ is the time step, $x\in\Re^{n_x}$ is the state, $u\in\Re^{n_u}$ is the control input, and $w\in\Re^{n_w}$ is a zero-mean white Gaussian noise with unit covariance, that is, 
\begin{equation}
\mathbb{E}\left[w_k\right] = 0, \qquad \qquad
\mathbb{E}\left[w_{k_1}{w_{k_2}^{\top}}\right] = 
\begin{cases} I_{n_w}, &\mbox{if } k_1 = k_2,\\
0, &\mbox{otherwise.}
\end{cases}\label{eq:GaussianNoise}
\end{equation}
We also assume that
\begin{equation}
\mathbb{E}\left[x_{k_1}w_{k_2}^{\top}\right]=0,\qquad0 \leq k_1 \leq k_2 \leq N.
\end{equation}
The initial state $x_0$ is a random vector that is drawn from the multi-variate normal distribution
\begin{equation}
x_0\sim\mathcal{N}(\mu_0,\Sigma_0),\label{eq:x0}
\end{equation}
where $\mu_0 \in \Re^{n_x}$ is the initial state mean and $\Sigma_0 \in \Re^{n_x \times n_x}$ is the initial state covariance. 
We assume that $\Sigma_0 \succeq 0$. 
Our objective is to steer the trajectories of the system~(\ref{eq:SystemDynamics}) from this initial distribution to the terminal Gaussian distribution
\begin{equation}  \label{eq:xN}
x_N\sim\mathcal{N}(\mu_N,\Sigma_N),
\end{equation}
where $\mu_N \in \Re^{n_x}$ and $\Sigma_N \in \Re^{n_x \times n_x}$ with $\Sigma_N \succ 0$, at a given time $N$, while minimizing the cost function
\begin{equation}  \label{eq:originalObjFunc}
J(x_0,\ldots,x_{N-1},u_0,\ldots,u_{N-1}) = \mathbb{E}\left[\sum_{k=0}^{N-1}x_k^\top Q_k x_k + u_k^\top R_k u_k\right],
\end{equation}
where $Q_k \succeq 0$ and $R_k \succ 0$ for all $k=0,1,\ldots,N-1$.

Note that (\ref{eq:originalObjFunc}) does not include a terminal cost owing to the terminal constraint~(\ref{eq:xN}).
The objective is to compute the optimal control input, which ensures that the probability of  the state violation at any given time is below a pre-specified threshold, say, 
\begin{equation}  \label{eq:originalChanceConstraints}
\texttt{Pr}(x_k \notin \chi) \leq P_{\rm{fail}},\qquad k=0,\ldots,N, 
\end{equation}
where $\texttt{Pr}()$ denotes the probability of an event, $\chi \subset \Re^{n_x}$ is the state constraint set, and $P_{\rm{fail}} \in [0, 1]$ is the threshold for the probability of failure. 
Note that $\mu_0$, $\mu_N$, $\Sigma_0$, and $\Sigma_N$ need to be designed to satisfy this chance constraint.
Optimization problems with these types of constraints are known as the \emph{chance-constrained} optimization problems~\cite{mesbah2016stochastic}.
In this work, we assume for simplicity  that $\chi$ is convex, but chance-constraints with non-convex constraints are also possible (see, for instance,~\cite{ono2010chance})

It is assumed that the system (\ref{eq:SystemDynamics}) is controllable, that is $x_N$ is reachable for any $x_N \in \Re^{n_x}$, provided that $w_k = 0$ for $k \in [0, N-1]$. This implies that given any $x_N \in \Re^{n_x}$ and $x_0 \in \Re^{n_x}$, there exists a sequence of control inputs $\{u_k\}_{k=0}^{N-1}$ that steers $x_0$ to $x_N$.

\subsection{Preliminaries}

We provide an alternative description of the system dynamics in (\ref{eq:SystemDynamics}) that will be instrumental for solving the problem.
The discussion below is borrowed from~\cite{goldshtein2017finite}.

At each time step $k$, we explicitly compute the system state $x_k$ as follows. 
Let $A_{k_1,k_0}$, $B_{k_1,k_0}$, and $D_{k_1,k_0}$, where $k_1>k_0$, denote the transition matrices of the state, input, and the noise term from step $k_0$ to step $k_1+1$, respectively, as follows
\begin{subequations}
	\begin{align}
	&A_{k_1,k_0} = A_{k_{1}}A_{k_{1}-1}\cdots A_{k_{0}},\\
	&B_{k_1,k_0} = A_{k_{1},k_{0}+1}B_{k_0},\\
	&D_{k_1,k_0} = A_{k_{1},k_{0}+1}D_{k_0}.
	\end{align}
\end{subequations}
We define the augmented vectors $U_{k} \in \Re^{(k+1)n_u}$ and $W_{k} \in \Re^{(k+1)n_w}$ as
\begin{equation}
U_{k} = \begin{bmatrix}
u_{0}\\u_{1}\\ \vdots \\ u_{k}
\end{bmatrix},
\qquad
W_{k} = \begin{bmatrix}
w_{0}\\w_{1}\\ \vdots \\ w_{k}
\end{bmatrix}.
\end{equation}
Then $x_k$ can be equivalently computed from
\begin{equation}
x_k = \bar{A}_k x_0 + \bar{B}_k U_k + \bar{D}_k W_k,\label{eq:SystemDynamicsConcatenated}
\end{equation}
where 
\begin{subequations}
	\begin{align}
	&\bar{A}_k = A_{k-1,0},\\
	&\bar{B}_k = \begin{bmatrix}
	B_{k-1,0} & B_{k-1,1} & \cdots & B_{k-1}
	\end{bmatrix},\\
	&\bar{D}_k = \begin{bmatrix}
	D_{k-1,0} & D_{k-1,1} & \cdots & D_{k-1}
	\end{bmatrix}.
	\end{align}
\end{subequations}
Furthermore, we introduce the augmented state vector $X \in \Re^{(k+1) n_x}$ as follows
\begin{equation}
X_k=\begin{bmatrix}
x_{0}\\x_{1}\\\vdots\\x_{k} 
\end{bmatrix}.
\end{equation}
It follows that the system dynamics~(\ref{eq:SystemDynamics}) take the equivalent form
\begin{equation}
X = \mathcal{A}x_{0} + \mathcal{B}U+\mathcal{D}W, \label{eq:convertedStateDynamics}
\end{equation}
where $X = X_{N}\in \Re^{(N+1) n_x} $, $U = U_{N-1}\in\Re^{Nn_u}$, and $W = W_{N-1}\in \Re^{Nn_w}$,
and the matrices $\mathcal{A}\in\Re^{(N+1)n_x\times n_x}$, $\mathcal{B}\in\Re^{(N+1)n_x\times Nn_u}$, and $\mathcal{D}\in\Re^{(N+1)n_x\times Nn_w}$ are defined as
\begin{subequations}
 \begin{equation}
 \mathcal{A} = \begin{bmatrix}
 I \\
 \bar{A}_1\\
 \bar{A}_2\\
 \vdots\\
 \bar{A}_{N}
 \end{bmatrix},
\qquad\qquad
 \mathcal{B} = \begin{bmatrix}
 0 & 0 & \cdots & 0\\
 B_0& 0 & \cdots & 0\\
 B_{1,0}& B_1 & \cdots & 0\\
 \vdots & \vdots & \vdots & \vdots\\
 B_{N-1,0}& B_{N-1,1} & \cdots & B_{N-1}
 \end{bmatrix},
 \end{equation}
 and
 \begin{equation}
\mathcal{D} = \begin{bmatrix}
0 & 0 & \cdots & 0\\
D_0& 0 & \cdots & 0\\
D_{1,0}& D_1 & \cdots & 0\\
\vdots & \vdots & \vdots & \vdots\\
D_{N-1,0}& D_{N-1,1} & \cdots & D_{N-1}
\end{bmatrix}.
\end{equation}
\end{subequations}
Note that 
\begin{subequations}
	\begin{align}  \label{eq:x0x0x0WWW}
		\mathbb{E}[x_0x_0^\top] &= \Sigma_0 + \mu_0\mu_0^\top,\\
		\mathbb{E}[x_0W^\top] &= 0, \\
		\mathbb{E}[WW^\top] &= I_{Nn_w}.
	\end{align}
\end{subequations}

Using the previous expressions for $X$ and $U$, we may rewrite the objective function in ~(\ref{eq:originalObjFunc}) as follows 
\begin{equation}
J(X,U) = \mathbb{E}\left[X^\top \bar{Q} X + U^\top \bar{R} U \right],\label{eq:convertedObjFunc}
\end{equation}
where $\bar{Q} = \texttt{blkdiag}(Q_0,Q_1,\ldots,Q_{N-1},0) \in \Re^{(N+1)n_x\times(N+1)n_x}$ and $\bar{R} = \texttt{blkdiag}(R_0,R_1,\ldots,R_{N-1}) \in \Re^{Nn_u\times Nn_u}$. 
Note that, since $Q_k\succeq0$ and $R_k\succ0$ for all $k = 0,1,\ldots,N-1 $, it follows that $\bar{Q}\succeq0$ and $\bar{R}\succ0$.

The boundary conditions~(\ref{eq:x0}) and (\ref{eq:xN}) also take the form 
\begin{subequations}
	\begin{align}
		\mu_0 &= E_0\mathbb{E}[X],\label{eq:mu0}\\
		\Sigma_0 &= E_0\left(\mathbb{E}[XX^\top] - \mathbb{E}[X]\mathbb{E}[X]^\top\right) E_0^\top,       \label{eq:Sigma0}
	\end{align}\label{eq:X0}
\end{subequations}
and
\begin{subequations}
	\begin{align}
	\mu_N &= E_N\mathbb{E}[X],\label{eq:muN}\\
	\Sigma_N &= E_N\left(\mathbb{E}[XX^\top] - \mathbb{E}[X]\mathbb{E}[X]^\top\right) E_N^\top,		\label{eq:SigmaN}
	\end{align}\label{eq:XN}
\end{subequations}
where  $E_0 \triangleq \left[I_{n_x},0,\ldots,0\right]\in \Re^{n_x\times(N+1)n_x}$ and $E_N \triangleq \left[0,\ldots,0,I_{n_x}\right]\in \Re^{n_x\times(N+1)n_x}$, respectively. 

Finally,  the chance constraints~(\ref{eq:originalChanceConstraints}) can be rewritten as
\begin{equation}  \label{eq:convertedChanceConstraints}
\texttt{Pr}(X \notin \mathcal{X}) \leq P_{\rm{fail}}, 
\end{equation}
where $\mathcal{X} \subset \Re^{(N+1)n_x}$ is a convex set.

The objective of this paper is to solve the following problem.

\begin{Problem}\label{prob:OriginalProblem}
Given the system (\ref{eq:convertedStateDynamics}), find the control sequence $U^\ast$ that minimizes 
the cost function Eq.~(\ref{eq:convertedObjFunc})  subject to the initial state constraints~(\ref{eq:X0}), the terminal state constraints~(\ref{eq:XN}), 
and the chance constraint~(\ref{eq:convertedChanceConstraints}).
\end{Problem}

In Section~{\ref{sec:ProposedApproach} we show how to solve Problem~\ref{prob:OriginalProblem} by converting it to a convex programming problem.
Before doing that, we first investigate the case without chance constraints.
The investigation of this case will clarify the difference of our work from prior research works on covariance steering, where the mean of the initial and terminal Gaussian distribution constraints are assumed to be zero. 

\section{No Chance Constraint Case \label{sec:NoChanceConstraint}}

Before discussing the general case with chance constraints, in this section we introduce the case without chance constraints and show that, similarly to the work by Goldshtein and Tsiotras \cite{goldshtein2017finite}, where the authors considered the case with minimal control effort, 
it is possible to separately solve the mean and the covariance steering optimization problems, 
even with the more general $\ell_2$-norm objective function of equation~(\ref{eq:originalObjFunc}).

\subsection{Separation of Mean and Covariance Problems}

In \cite{goldshtein2017finite} the authors showed that it is possible to separate the mean and the covariance evolutions of the system.
Along with a similar separation of the cost, this observation allowed to 
independently design a \emph{mean-steering} and a \emph{covariance-steering} controller independently from one another. 
In this section we show that a similar result  holds for the cost~(\ref{eq:originalObjFunc}). 
While the objective function in \cite{goldshtein2017finite} had only a control penalty term
\begin{equation} \label{eq:minimumEffortCost}
	J(U) = \mathbb{E}[\sum_{k=0}^{N-1}u_k^\top u_k],
\end{equation}
in this paper we work with the more general cost~(\ref{eq:originalObjFunc}).
First, it follows immediately from Eq.~(\ref{eq:SystemDynamicsConcatenated}) that
	\begin{equation}  \label{eq:mean}
	\mu_k \triangleq \mathbb{E}[x_k] = \bar{A}_k \mu_0 + \bar{B}_k \bar{U}_k,
	\end{equation}
	where $\bar{U}_k = \mathbb{E}[U_k]$. 
	Furthermore, by defining
	\begin{equation}
	\tilde{U}_k \triangleq U_k - \bar{U}_k, \qquad 
	\tilde{x}_k \triangleq x_k - \mu_k, 
	\end{equation}
	and using~(\ref{eq:SystemDynamicsConcatenated}), the following equation holds for $\tilde{x}_k$
	\begin{equation} \label{eq:xtilde}
	\tilde{x}_k = \bar{A}_k\tilde{x}_0 + \bar{B}_k\tilde{U}_k+\bar{D}_k W_k.
	\end{equation}
Furthermore,
	\begin{subequations}
		\begin{align}
		\Sigma_k &\triangleq \mathbb{E}[\tilde{x}_k\tilde{x}_k^\top],\\
		&= \mathbb{E}\left[\big(\bar{A}_k\tilde{x}_0 + \bar{B}_k\tilde{U}_k+\bar{D}_kW_k\right)\left(\bar{A}_k\tilde{x}_0 + \bar{B}_k\tilde{U}_k+\bar{D}_kW_k\big)^\top
		\right],\\
		&= \bar{A}_k\mathbb{E}[\tilde{x}_0\tilde{x}_0^\top]\bar{A}_k^\top
		+ \bar{A}_k\mathbb{E}[\tilde{x}_0\tilde{U}_k^\top]\bar{B}_k^\top
		+ \bar{B}_k\mathbb{E}[\tilde{U}_k\tilde{x}_0^\top]\bar{A}_k^\top
		+ \bar{B}_k\mathbb{E}[\tilde{U}_k\tilde{U}_k^\top]\bar{B}_k^\top \nonumber\\
		&\hspace{2cm}
		+ \bar{D}_k\mathbb{E}[{W}_k{W}_k^\top]\bar{D}_k^\top
		+ \bar{D}_{k-1}\mathbb{E}[{W}_{k-1}\tilde{U}_k^\top]\bar{B}_k^\top
		+ \bar{B}_k\mathbb{E}[\tilde{U}_k {W}_{k-1}^\top]\bar{D}_{k-1}^\top.
		\end{align}
	\end{subequations}
	Note  that the evolution of the mean $\mu_k$ from (\ref{eq:mean}) depends only on $\bar{U}_k$, whereas the evolution of $\tilde{x}_k$ and  $\Sigma_k$ depend solely on $\tilde{U}_k$ and $W_k$. 
It follows from Eq.~(\ref{eq:convertedStateDynamics}) and (\ref{eq:mean}) that
\begin{equation}  \label{eq:meanDynamics}
	\bar{X} \triangleq \mathbb{E}[X] = \mathcal{A}\mu_0 + \mathcal{B}\bar{U},
\end{equation}
and from (\ref{eq:xtilde}) that
\begin{equation} \label{eq:covarianceDynamics}
	\tilde{X} \triangleq X - \mathbb{E}[X] = \mathcal{A}\tilde{x}_0 + \mathcal{B}\tilde{U} + \mathcal{D}W.
\end{equation}
	The objective function~(\ref{eq:convertedObjFunc}) can also be rewritten as follows
	\begin{subequations}
		\begin{align}
	J(X,U) 
	&= \mathbb{E}\left[X^\top \!\bar{Q} X + U^\top \!\bar{R} U\right]\\
		&= \mathtt{tr}\big(\bar{Q}\,\mathbb{E}[\tilde{X}\tilde{X}^\top]\big)
		+ \bar{X}^\top\!\bar{Q} \bar{X} 
		+ \mathtt{tr}\big(\bar{R}\mathbb{E}[\tilde{U}\tilde{U}^\top]\big)
		+ \bar{U}^\top\!\bar{R} \bar{U},\\
		&=J_\mu(\bar{X},\bar{U}) + J_\Sigma(\tilde{X},\tilde{U}),
		\end{align}
	\end{subequations}
	where
	\begin{equation}
	J_\mu(\bar{X},\bar{U}) = \bar{X}^\top \bar{Q} \bar{X} + \bar{U}^\top \bar{R} \bar{U},		\label{eq:Jmu}
	\end{equation}
	and
	\begin{equation}
	J_\Sigma(\tilde{X},\tilde{U}) =  \mathtt{tr}\big(\bar{Q}\mathbb{E}[\tilde{X}\tilde{X}^\top]\big)
	+ \mathtt{tr}\big(\bar{R}\mathbb{E}[\tilde{U}\tilde{U}^\top]\big), \label{eq:JSigma}
	\end{equation}
and	where $\texttt{tr}()$ denotes the trace of a matrix. 
	It follows that the original optimization problem in terms of $(X,U)$ is equivalent to two separate optimization problems in terms of $(\bar{X},\bar{U})$
and $(\tilde{X},\tilde{U})$ with optimization costs (\ref{eq:Jmu}) and (\ref{eq:JSigma}), respectively.

We have therefore shown the following result.

\begin{Proposition}
Let the system (\ref{eq:convertedStateDynamics}), the initial and terminal state constraints (\ref{eq:x0}) and (\ref{eq:xN}), and the objective function (\ref{eq:originalObjFunc}). 
The control sequence $U^*$ that solves this optimization problem is given by  $U^* = \bar{U}^* + \tilde{U}^*$, where $\bar{U}^*$ 
solves the optimization problem 
\begin{equation} \label{MeanSteerProb}
\text{Mean Steering}
\left\{
\begin{aligned}
\min_{(\bar{X},\bar{U})} J_\mu(\bar{X},\bar{U}) = \bar{X}^\top &\bar{Q} \bar{X} + \bar{U}^\top \bar{R} \bar{U}\\
\mathrm{subject~to}~~\bar{X} = \mathcal{A}\mu_0 &+ \mathcal{B}\bar{U},\\
E_0 \bar{X}  = \mu_0,~~& E_N \bar{X} = \mu_N,
\end{aligned}
\right.
\end{equation}
and $\tilde{U}^*$ solves the optimization problem
\begin{equation}   \label{CovSteerProb}
\text{Covariance Steering}
\left\{
\begin{aligned}
\min_{(\tilde{X},\tilde{U})} J_\Sigma(\tilde{X},\tilde{U}) =  \mathtt{tr}\big(\bar{Q}\mathbb{E}[&\tilde{X}\tilde{X}^\top]\big)
	+ \mathtt{tr}\big (\bar{R}\mathbb{E}[\tilde{U}\tilde{U}^\top]\big),\\
\mathrm{subject~to}~~\tilde{X} \triangleq X - \mathbb{E}[X] &= \mathcal{A}\tilde{x}_0 + \mathcal{B}\tilde{U} + \mathcal{D}W\\
E_0 \tilde{X}\tilde{X}^\top E_0^\top = \Sigma_0,&~~ E_N \tilde{X}\tilde{X}^\top E_N^\top = \Sigma_N. 
\end{aligned}
\right.
\end{equation}

\end{Proposition}
The rest of this section introduces the methods to solve these two subproblems.

\subsection{Optimal Mean Steering}

The solution to the optimal mean steering subproblem is summarized in the following proposition.

\begin{Proposition}
The optimal control sequence that solves
 the optimization problem (\ref{MeanSteerProb}) 
is given by
	\begin{equation}
	\bar{U}^\ast = \mathcal{R}^{-1}\left(\mathcal{B}^\top\bar{Q}\mathcal{A}\mu_0+ \bar{B}_N^\top(\bar{B}_N\mathcal{R}^{-1}\bar{B}_N^\top)^{-1}
	\left(\mu_N-\bar{A}_N\mu_0-\bar{B}_N\mathcal{R}^{-1}\mathcal{B}^\top\bar{Q}\mathcal{A}\mu_0\right)\right),		\label{eq:meanSteerController}
	\end{equation}
	where $\mathcal{R} = (\mathcal{B}^\top\bar{Q}\mathcal{B}+\bar{R})$.
\end{Proposition}

\begin{proof}
Since the terminal constraint is $\mu_N = E_N \bar{X} = \bar{A}_N\mu_0 + \bar{B}_N\bar{U}$
	we can write the Lagrangian as 
	\begin{align}
		\mathcal{L}(\bar{U},\lambda) &= \bar{X}^\top\bar{Q}\bar{X} + \bar{U}^\top \bar{R}\bar{U} + \lambda^\top(\mu_N - \bar{A}_N\mu_0 - \bar{B}_N\bar{U})\\
		&= (\mathcal{A}\mu_0 + \mathcal{B}\bar{U})^\top\bar{Q}(\mathcal{A}\mu_0 + \mathcal{B}\bar{U}) + \bar{U}^\top \bar{R}\bar{U} + \lambda^\top(\mu_N - \bar{A}_N\mu_0 - \bar{B}_N\bar{U}),
	\end{align}
	where $\lambda\in\Re^{n_x}$. 
	The first-order optimality condition yields
	\begin{equation}
		\nabla_{\bar{U}}\mathcal{L}=2(\mathcal{B}^\top\bar{Q}\mathcal{B}+\bar{R})\bar{U} + 2\mathcal{B}^\top\bar{Q}\mathcal{A}\mu_0 - \bar{B}_N^\top\lambda = 0.
	\end{equation}
	Thus,
	\begin{equation}
		\bar{U}^* = \mathcal{R}^{-1}(\mathcal{B}^\top\bar{Q}\mathcal{A}\mu_0+ \frac{1}{2}\bar{B}_N^\top\lambda),\label{eq:EUwlambda}
	\end{equation}
	where $\mathcal{R} = (\mathcal{B}^\top\bar{Q}\mathcal{B}+\bar{R})$ is invertible because of the second-order optimality condition
	\begin{equation}
	\nabla_{\bar{U}\bar{U}}\mathcal{L} = \mathcal{B}^\top\bar{Q}\mathcal{B}+\bar{R} \succ 0.
	\end{equation}
	
In order  to find the optimal value of $\lambda$
we substitute equation~(\ref{eq:EUwlambda}) into the terminal constraint  to obtain
	\begin{equation}
			\mu_N = \bar{A}_N\mu_0 + \bar{B}_N\mathcal{R}^{-1}(\mathcal{B}^\top\bar{Q}\mathcal{A}\mu_0+ \frac{1}{2}\bar{B}_N^\top\lambda),
	\end{equation}
	or
	\begin{equation}
	\frac{1}{2}\bar{B}_N\mathcal{R}^{-1}\bar{B}_N^\top\lambda = \mu_N - \bar{A}_N\mu_0 - \bar{B}_N\mathcal{R}^{-1}\mathcal{B}^\top\bar{Q}\mathcal{A}\mu_0.
	\end{equation}
Note that $\mathrm{rank}(\bar{B}_N\mathcal{R}^{-1}\bar{B}_N^\top) = \mathrm{rank}(\mathcal{R}^{-1/2}\bar{B}_N^\top)$. 
Also, since the system is controllable, it follows that $\mathrm{rank}(\bar{B}_N)$ is full row rank, that is, $\mathrm{rank}(\bar{B}_N) = n_x$ \cite{goldshtein2017finite}.
In addition, since $\mathcal{R}$ is invertible, $\mathrm{rank}(\mathcal{R}^{-1/2}) = Nn_u$. 
It follows from Corollary 2.5.10 in \cite{bernstein2009matrix} that
\begin{subequations}
	\begin{align}
		\mathrm{rank}(\mathcal{R}^{-1/2}) + \mathrm{rank}(\bar{B}_N^\top)-Nn_u \leq &\mathrm{rank}(\mathcal{R}^{-1/2}\bar{B}_N^\top) \leq \mathtt{min}\left\{\mathrm{rank}(\mathcal{R}^{-1/2}), \mathrm{rank}(\bar{B}_N^\top)\right\}\\
	    n_x \leq &\mathrm{rank}(\bar{B}_N\mathcal{R}^{-1}\bar{B}_N^\top) \leq \mathtt{min}\left\{Nn_u, n_x\right\} = n_x
	\end{align}
\end{subequations}
Thus, the matrix $(\bar{B}_N\mathcal{R}^{-1}\bar{B}_N^\top)$ is full rank and invertible.
Therefore,
	\begin{equation}
		\lambda = 2(\bar{B}_N\mathcal{R}^{-1}\bar{B}_N^\top)^{-1}
		\left(\mu_N-\bar{A}_N\mu_0-\bar{B}_N\mathcal{R}^{-1}\mathcal{B}^\top\bar{Q}\mathcal{A}\mu_0\right).
	\end{equation}
By substituting in (\ref{eq:EUwlambda})  the expression for the optimal mean steering controller, the experssion~(\ref{eq:meanSteerController}) follows.
\end{proof}

By comparing (\ref{eq:meanSteerController}) with the corresponding controller in  \cite{goldshtein2017finite} we have the following immediate result.

\begin{Corollary}
	The minimum-effort mean-steering optimal controller introduced in \cite{goldshtein2017finite} is a special case of the optimal controller~(\ref{eq:meanSteerController}). 
\end{Corollary}

\begin{proof}
	The result directly follows by plugging $\bar{Q} = 0$, $\bar{R} = I$ into Eq.~(\ref{eq:meanSteerController}).
\end{proof}

\subsection{Optimal Covariance Steering}\label{sec:CSOnly} 

While many previous works have attempted to solve the optimal covariance-steering problem, the majority of them solve this problem subject to 
a minimum effort cost function as in (\ref{eq:minimumEffortCost}).
Bakolas \cite{bakolas2016optimalCDC} addressed the case with the more general $L_2$-norm cost function Eq.~(\ref{eq:originalObjFunc}) by using a convex relaxation to change the terminal constraint to an inequality as follows
\begin{equation}  \label{eq:SigmaNleqSigamF}
	E_N\left(\mathbb{E}[XX^\top] - \mathbb{E}[X]\mathbb{E}[X]^\top\right) E_N^\top \preceq \Sigma_N. 
\end{equation}
By  making the problem convex, it can be efficiently solved using standard convex programming solvers. 
At the same time, but independently, Halder and Wendel~\cite{halder2016finite} solved a problem with a similar terminal covariance constraint using a soft constraint on the terminal state covariance under continuous-time dynamics. 

\section{Chance Constrained Case} \label{sec:ProposedApproach}

This section introduces the proposed approach to solve the covariance steering problem with chance constraints in Problem~\ref{prob:OriginalProblem}.
\subsection{Proposed Approach}
First, we assume that, at each time step, the control input is represented as follows
\begin{equation}
	u_k = \ell_k \left[\mathbbm{1}_{n_x}^\top,x_0^\top,x_1^\top,\ldots,x_k^\top\right]^\top,
\end{equation}
where $\mathbbm{1}_{n_x} = [1,\ldots,1]^\top \in \Re^{n_x}$ and $\ell_k\in \Re^{n_u\times n_x(k+2)}$. 
Thus, we may write the relationship between $X$ and $U$ as follows
\begin{equation}
U=L \boldX,
\end{equation}
where $\boldX = [\mathbbm{1}_{n_x}^\top, X^\top]^\top \in \Re^{(N+2)n_x}$ is the augmented state sequence until step $N$ and $L \in \Re^{Nn_u \times (N+2)n_x}$ is the control gain matrix.
In order to ensure that the control input at time step $k$ depends only on $x_i$ for $i=0,1,\ldots,k$
(so that the control input $U$ is causally related to the state history, that is, it is non-anticipative)
the matrix $L$ has to be of the form
\begin{equation}
	L = [L_{\mathbbm{1}},L_X],\label{eq:LL1LX}
\end{equation}
where $L_{\mathbbm{1}} \in \Re^{Nn_u\times n_x}$ and $L_{X} \in \Re^{Nn_u\times (N+1)n_x}$ is a lower block triangular matrix. 
Using this $L$ we convert the problem from finding the optimal control input sequence $U^\ast$ to finding the optimal control gain matrix $L^\ast$.
It follows from~(\ref{eq:convertedStateDynamics}) that 
	\begin{equation}
	\boldX = \begin{bmatrix}I_{n_x} & 0\\0 & \mathcal{A}\end{bmatrix}\begin{bmatrix}\mathbbm{1}_{n_x}\\x_0\end{bmatrix} + \begin{bmatrix}
	0\\\mathcal{B}\end{bmatrix}L\boldX + 
	\begin{bmatrix}0\\\mathcal{D}\end{bmatrix}W,
\end{equation}
and hence
\begin{equation}   \label{eq:X}
	\boldX = (I-\bm{\mathcal{B}}L)^{-1}(\bm{\mathcal{A}}\boldX_0 + \bm{\mathcal{D}}W),
\end{equation}
where $\boldX_0 = [\mathbbm{1}_{n_x}^\top,~x_0^\top]^\top \in \Re^{2n_x}$, $\bm{\mathcal{A}} = \texttt{blkdiag}(I_{n_x},\mathcal{A}) \in \Re^{(N+2)n_x\times2n_x}$, $\bm{\mathcal{B}} = [0,~\mathcal{B}^\top]^\top \in \Re^{(N+2)n_x\times Nn_u}$, $\bm{\mathcal{D}} = [0,~\mathcal{D}^\top]^\top \in \Re^{(N+2)n_x\times Nn_w}$.
Note that $(I-\bm{\mathcal{B}}L)$ is invertible because 
\begin{equation}   \label{eq:00BL1BLX}
\bm{\mathcal{B}}L = \begin{bmatrix}0\\\mathcal{B}\end{bmatrix}[L_{\mathbbm{1}},L_X],
			   = \begin{bmatrix}0 & 0\\\mathcal{B}L_{\mathbbm{1}} & \mathcal{B}L_X\end{bmatrix}.
\end{equation}
Since $\mathcal{B}L_X$ is strictly lower-block triangular, $\bm{\mathcal{B}}L$ is also strictly lower-block triangular\footnote{Here, a strictly lower-block triangular matrix is a lower-block triangular matrix with zero matrices on its diagonal elements.}. 

Using $\boldX$ from (\ref{eq:X}) the objective function~(\ref{eq:convertedObjFunc}) can be written as 
\begin{equation}
	J(L) = \mathbb{E}\left[(\bm{\mathcal{A}} \boldX_0 + \bm{\mathcal{D}}W)^\top(I-\bm{\mathcal{B}}L)^{-\top} \bar{\bm{Q}} (I-\bm{\mathcal{B}}L)^{-1}(\bm{\mathcal{A}} \boldX_0 + \bm{\mathcal{D}}W) + \boldX^\top L^\top \bar{R} L\boldX \right],\label{eq:LObjFunc}
\end{equation}
where $\bar{\bm{Q}} = \texttt{blkdiag}(0,\bar{Q}) \in \Re^{(N+2)n_x\times(N+2)n_x}$. Note that $\bar{\bm{Q}} \succeq 0$.

Similarly to Bakolas~\cite{bakolas2016optimalCDC}, we introduce the decision variable $K$ such that
\begin{equation}
	K\triangleq L(I-\bm{\mathcal{B}}L)^{-1}.\label{eq:Kdef}
\end{equation}
It follows that $I+\bm{\mathcal{B}}K = (I - \bm{\mathcal{B}}L)^{-1}$.
Then, $\boldX$ and $U$ are rewritten as
\begin{equation}   \label{eq:X=IpBKAx0Dw}
\boldX = (I+\bm{\mathcal{B}}K)(\bm{\mathcal{A}}\boldX_0 + \bm{\mathcal{D}}W),
\end{equation}
and
\begin{equation}  \label{eq:U=AxDw}
U = K(\bm{\mathcal{A}}\boldX_0 + \bm{\mathcal{D}}W).
\end{equation}

Before continuing, we show that $K$ defined in (\ref{eq:Kdef}) is lower block triangular.
This ensures that the resulting $U$ is non-anticipative. 

\begin{Lemma}
	Let $L$ be defined as in Eq.~(\ref{eq:LL1LX}), let $\bm{\mathcal{B}}$ be a strictly lower block triangular matrix, and let $I$ be an identity matrix with proper dimensions. 
	Then, $K$ as defined in Eq.~(\ref{eq:Kdef}) is represented as 
	$ K = \begin{bmatrix}K_{\mathbbm{1}}&K_X\end{bmatrix}$,
	where $K_{\mathbbm{1}} \in \Re^{Nn_u\times n_x}$ and $K_{X} \in \Re^{Nn_u\times (N+1)n_x}$ is lower block triangular. 
\end{Lemma}

\begin{proof}
	Recall from equation~(\ref{eq:00BL1BLX}) that $\bm{\mathcal{B}}L$ is a strictly lower block triangular matrix.
	Thus, $I-\bm{\mathcal{B}}L$ is lower block triangular and invertible
	\begin{equation}
		(I-\bm{\mathcal{B}}L)^{-1} = \begin{bmatrix}
		I & 0 \\ -\mathcal{B}L_{\mathbbm{1}} & I-\mathcal{B}L_X
		\end{bmatrix}^{-1}\\
		= \begin{bmatrix}
		I & 0 \\ (I-\mathcal{B}L_X)^{-1}\mathcal{B}L_{\mathbbm{1}} & (I-\mathcal{B}L_X)^{-1}
		\end{bmatrix}.
	\end{equation}
	Thus,
	\begin{align}
	K &= [L_{\mathbbm{1}},L_X]\begin{bmatrix}I & 0 \\ (I-\mathcal{B}L_X)^{-1}\mathcal{B}L_{\mathbbm{1}} & (I-\mathcal{B}L_X)^{-1}\end{bmatrix}\\
	  &= \begin{bmatrix}
		  L_{\mathbbm{1}} + L_X(I-\mathcal{B}L_X)^{-1}\mathcal{B}L_{\mathbbm{1}} & L_X(I-\mathcal{B}L_X)^{-1}
	  \end{bmatrix}
	\end{align}
	As the inverse of lower-block triangular matrix is also lower block triangular, $(I-\mathcal{B}L_X)^{-1}$ is also lower block triangular. 
	The multiplication of two lower block triangular matrices $L_X$ and $(I-\mathcal{B}L_X)^{-1}$ yields a matrix that is also lower block triangular, and hence $K$ is represented as 
	\begin{equation}\label{eq:KK1KX}
		K = \begin{bmatrix}K_{\mathbbm{1}}&K_X\end{bmatrix},
	\end{equation}
	where $K_{\mathbbm{1}} \in \Re^{Nn_u\times n_x}$ and $K_{X} \in \Re^{Nn_u\times (N+1)n_x}$ is lower block triangular. 

\end{proof}

We may now prove the following result.
 
\begin{Proposition}
	Let $\boldX$ as in~(\ref{eq:X=IpBKAx0Dw}), the control input  as in~(\ref{eq:U=AxDw}), the objective function~(\ref{eq:convertedObjFunc}) and the boundary conditions
	\begin{equation}\label{eq:X0augmented}
		\boldsymbol{\mu}_0 = \begin{bmatrix}\mathbbm{1}_{n_x}\\ \mu_0 \end{bmatrix},\qquad \mathbf{\Sigma}_0 = \begin{bmatrix}
		0_{n_x} & 0_{n_x}\\0_{n_x} & \Sigma_0
		\end{bmatrix} \succeq 0.
	\end{equation} 
	Then, the objective function takes the form
	\begin{equation}
		J(K) = \mathtt{tr}\left(\left(\left(I+\bm{\mathcal{B}}K\right)^{\top} \bar{\bm{Q}} \left(I+\bm{\mathcal{B}}K\right) + K^\top \bar{R} K\right)\left(\bm{\mathcal{A}}\left(\boldsymbol{\mu_0 \mu_0}^\top + \mathbf{\Sigma}_0\right)\bm{\mathcal{A}}^\top + \bm{\mathcal{DD}}^\top\right)\right).\label{eq:Kobj2Solve}
	\end{equation}
which is a quadratic expression in $K$.
\end{Proposition}

\begin{proof}
	Using equations~(\ref{eq:X=IpBKAx0Dw}) and (\ref{eq:U=AxDw}) the objective function~(\ref{eq:convertedObjFunc}) can be written as follows
	\begin{align}
		J(K) = \mathbb{E}\left[(\bm{\mathcal{A}}\bm{x}_0 + \bm{\mathcal{D}}W)^\top(I+\bm{\mathcal{B}}K)^{\top} \bar{\bm{Q}} (I+\bm{\mathcal{B}}K)(\bm{\mathcal{A}}\bm{x}_0 + \bm{\mathcal{D}}W)+
		(\bm{\mathcal{A}}\bm{x}_0 + \bm{\mathcal{D}}W)^\top K^\top \bar{R} K(\bm{\mathcal{A}}\bm{x}_0 + \bm{\mathcal{D}}W) \right].
	\end{align}
Equivalently,
	\begin{equation}
	\begin{aligned}
	J(K) &= \mathtt{tr} (\mathbb{E} [(\bm{\mathcal{A}}\bm{x}_0 + \bm{\mathcal{D}}W)^\top(I+\bm{\mathcal{B}}K)^{\top} \bar{\bm{Q}} (I+\bm{\mathcal{B}}K)(\bm{\mathcal{A}}\bm{x}_0+\bm{\mathcal{D}}W)\\
	&\hspace{4cm}+(\bm{\mathcal{A}}\bm{x}_0 + \bm{\mathcal{D}}W)^\top K^\top \bar{R} K(\bm{\mathcal{A}}\bm{x}_0 + \bm{\mathcal{D}}W) ]),
	\end{aligned}
	\end{equation}
Under cyclic permutations, the trace is invariant and hence
	\begin{subequations}
		\begin{align}
		J(K) &=\mathtt{tr}\left(\mathbb{E}\left[
		\left((I+\bm{\mathcal{B}}K)^{\top} \bar{\bm{Q}} (I+\bm{\mathcal{B}}K) +  K^\top \bar{R} K\right) (\bm{\mathcal{A}}\bm{x}_0 + \bm{\mathcal{D}}W)(\bm{\mathcal{A}}\bm{x}_0 + \bm{\mathcal{D}}W)^\top
		\right]\right),\\
		&= \mathtt{tr}\left(\left((I+\bm{\mathcal{B}}K)^{\top} \bar{\bm{Q}} (I+\bm{\mathcal{B}}K) +  K^\top \bar{R} K\right)
		\mathbb{E}\left[(\bm{\mathcal{A}}\bm{x}_0 + \bm{\mathcal{D}}W)(\bm{\mathcal{A}}\bm{x}_0 + \bm{\mathcal{D}}W)^\top
		\right]\right).
		\end{align} 
	\end{subequations}
Using the expression
	\begin{align}
	\mathbb{E}\left[(\bm{\mathcal{A}}\bm{x}_0 + \bm{\mathcal{D}}W)(\bm{\mathcal{A}}\bm{x}_0 + \bm{\mathcal{D}}W)^\top
	\right]&= \bm{\mathcal{A}}\mathbb{E}[\bm{x}_0\bm{x}_0^\top]\bm{\mathcal{A}}^\top +
	\bm{\mathcal{A}}\mathbb{E}[\bm{x}_0W^\top]\bm{\mathcal{D}}^\top +
	\bm{\mathcal{D}}\mathbb{E}[W\bm{x}_0^\top]\bm{\mathcal{A}}^\top + 
	\bm{\mathcal{D}}\mathbb{E}[WW^\top]\bm{\mathcal{D}}^\top,
	\end{align}
along with equation~(\ref{eq:x0x0x0WWW}) and 
\begin{equation}
\mathbb{E}[\bm{x}_0\bm{x}_0^\top] = \mathbf{\Sigma}_0 + \boldsymbol{\mu}_0\boldsymbol{\mu}_0^\top,\qquad\mathbb{E}[\bm{x}_0W^\top] = 0,
\end{equation}
we finally obtain~(\ref{eq:Kobj2Solve}).
\end{proof}

\subsection{Conversion of Chance Constraint to Deterministic Inequality Constraint}
In this section we show how to convert the chance constraint~(\ref{eq:convertedChanceConstraints}) to a form that is more amenable to computations.
To this end, we use the  approach in \cite{blackmore2009convex}. 
First, we assume that the feasible region $\mathcal{X}$ is defined as an intersection of $M$ linear inequality constraints as follows
\begin{equation}
	\mathcal{X} \triangleq \bigcap_{j=1}^M \{\boldX:\alpha_j^\top \boldX \leq \beta_j\},\label{eq:convertedX}
\end{equation}
where $\alpha_j \in \Re^{(N+2)n_x}$ and $\beta_j \in \Re$ with $j=1,2,\ldots,M$. 
Thus, the chance constraint~(\ref{eq:convertedChanceConstraints}) is converted to the condition 
\begin{subequations}
	\begin{align}
	\texttt{Pr}(\alpha_j^\top \boldX > \beta_j) &\leq p_j,\qquad j = 1,\ldots, M,\label{eq:conservativeEqj}\\
	\sum_{j=1}^{M}p_j &\leq P_{\rm{fail}}.	 \label{eq:sumOfP}
	\end{align}\label{eq:conservativeChanceConstraints}
\end{subequations}
Using the Boole-Bonferroni inequality~\cite{prekopa1988boole}, 
the authors of \cite{blackmore2009convex} showed that a feasible solution to the problem (\ref{eq:convertedX})-(\ref{eq:conservativeChanceConstraints}) is a feasible solution to the original chance-constrained problem.
Note that the constraint (\ref{eq:conservativeEqj}) can also be written as
\begin{equation} \label{eq:PrNew}
	\texttt{Pr}(\alpha_j^\top \boldX \leq \beta_j) \geq 1-p_j.
\end{equation}
As a result,  $\alpha_j^\top \boldX$ is a univariate Gaussian random variable such that $\alpha_j^\top \boldX \sim \mathcal{N}(\alpha_j^\top \bar{\boldX},\alpha_j^\top\Sigma_\boldX\alpha_j)$,
where 
\begin{equation}  \label{eq:barX}
\bar{\boldX} = \mathbb{E}[\boldX] = (I+\bm{\mathcal{B}}K)\bm{\mathcal{A}}\boldsymbol{\mu}_0,
\end{equation}
and
\begin{equation}
\begin{aligned} \label{eq:SigmaX}
\Sigma_\boldX &= \mathbb{E}\left[(\boldX-\bar{\boldX})(\boldX-\bar{\boldX})^\top\right]
	=\mathbb{E}\left[\left(\left(I+\bm{\mathcal{B}}K\right)\left(\bm{\mathcal{A}}(\bm{x}_0-\boldsymbol{\mu}_0)+\bm{\mathcal{D}}W\right)\right)\left(\left(I+\bm{\mathcal{B}}K\right)\left(\bm{\mathcal{A}}(\bm{x}_0-\boldsymbol{\mu}_0)+\bm{\mathcal{D}}W\right)\right)^{\top}\right],\\
	&=(I+\bm{\mathcal{B}}K)(\bm{\mathcal{A}}\mathbf{\Sigma}_0\bm{\mathcal{A}}^\top+\bm{\mathcal{D}}\bm{\mathcal{D}}^\top)(I+\bm{\mathcal{B}}K)^\top.
\end{aligned}
\end{equation}
It follows from inequality (\ref{eq:PrNew}) that
\begin{align}
	\mathtt{Pr}(\alpha_j^\top \boldX\leq\beta_j) &= \frac{1}{\sqrt{2\pi\alpha_j^\top\Sigma_\boldX\alpha_j}}\int_{-\infty}^{\beta_j}\exp\left(-\frac{(\xi-\alpha_j^\top \bar{\boldX})^2}{2\alpha_j^\top\Sigma_\boldX\alpha_j}\right)\mathrm{d}\xi,\\
	&=\Phi\left(\frac{\beta_j - \alpha_j^\top\bar{\boldX}}{\sqrt{\alpha_j^\top\Sigma_\boldX\alpha_j}}\right) \geq 1-p_j,
\end{align}
where $\Phi$ is the cumulative distribution function of the standard normal distribution, which is a monotonically increasing function. 
Thus, 
\begin{equation}
\frac{\beta_j - \alpha_j^\top\bar{\boldX}}{\sqrt{\alpha_j^\top\Sigma_\boldX\alpha_j}} \geq \Phi^{-1}\left(1-p_j\right),
\end{equation}
where $\Phi^{-1}$ is the inverse of $\Phi$.
Therefore, 
\begin{equation}   \label{eq:deterministicChanceConstraint}
\alpha_j^\top \bar{\boldX} - \beta_j + \sqrt{\alpha_j^\top\Sigma_\boldX\alpha_j}\, \Phi^{-1}\left(1-p_j\right)\leq 0.
\end{equation}
Previous  works~\cite{blackmore2009convex,blackmore2011chance,carvalho2014stochastic} assumed some prior knowledge about the covariance $\Sigma_\boldX$, enabling Eq.~(\ref{eq:deterministicChanceConstraint}) to be a linear inequality constraint.
However, as we are interested in the covariance steering problem, we cannot assume any prior knowledge in terms of $\Sigma_\boldX$. 
Thus, we prove the following result.
 
\begin{Theorem}
	Let $\bar{\boldX}$ as in~(\ref{eq:barX}), $\Sigma_\boldX$ as in~(\ref{eq:SigmaX}), and $\boldsymbol{\mu}_0$ and $\bm{\Sigma}_0 $ as in~(\ref{eq:X0augmented}). With the assumption, $\Sigma_0 \succeq 0$, the inequality constraint~(\ref{eq:deterministicChanceConstraint}) 	is converted to the following inequality constraint.
	\begin{equation}   \label{eq:deterministicChanceCosntraintBilinear}
	\alpha_j^\top(I+\bm{\mathcal{B}}K)\bm{\mathcal{A}}\boldsymbol{\mu}_0 - \beta_j + \|(\bm{\mathcal{A}}\mathbf{\Sigma}_0\bm{\mathcal{A}}^\top+\bm{\mathcal{D}}\bm{\mathcal{D}}^\top)^{1/2}(I+\bm{\mathcal{B}}K)^\top\alpha_j\| \, \Phi^{-1}\left(1-p_j\right) \leq 0.
	\end{equation}
\end{Theorem}

\begin{proof}
	Since $\Sigma_0 \succeq 0$, it follows that $\bm{\Sigma}_0 \succeq 0$ and
$	\bm{\mathcal{A}}\mathbf{\Sigma}_0\bm{\mathcal{A}}^\top+\bm{\mathcal{D}}\bm{\mathcal{D}}^\top \succeq 0$.
	Therefore, equation~(\ref{eq:SigmaX}) becomes
	\begin{equation}
		\Sigma_\boldX = (I+\bm{\mathcal{B}}K)(\bm{\mathcal{A}}\mathbf{\Sigma}_0\bm{\mathcal{A}}^\top+\bm{\mathcal{D}}\bm{\mathcal{D}}^\top)^{1/2}(\bm{\mathcal{A}}\mathbf{\Sigma}_0\bm{\mathcal{A}}^\top+\bm{\mathcal{D}}\bm{\mathcal{D}}^\top)^{1/2}(I+\bm{\mathcal{B}}K)^\top,
	\end{equation}
	and~(\ref{eq:deterministicChanceConstraint}) can be rewritten as
	\begin{equation}
		\alpha_j^\top\bar{\boldX}-\beta_j + \sqrt{\alpha_j^\top(I+\bm{\mathcal{B}}K)(\bm{\mathcal{A}}\mathbf{\Sigma}_0\bm{\mathcal{A}}^\top+\bm{\mathcal{D}}\bm{\mathcal{D}}^\top)^{1/2}(\bm{\mathcal{A}}\mathbf{\Sigma}_0\bm{\mathcal{A}}^\top+\bm{\mathcal{D}}\bm{\mathcal{D}}^\top)^{1/2}(I+\bm{\mathcal{B}}K)^\top\alpha_j} \, \Phi^{-1}\left(1-p_j\right)\leq 0,
	\end{equation}
	Note that since $(\bm{\mathcal{A}}\mathbf{\Sigma}_0\bm{\mathcal{A}}^\top+\bm{\mathcal{D}}\bm{\mathcal{D}}^\top)^{1/2}(I+\bm{\mathcal{B}}K)^\top\alpha_j$ is a vector, one obtains that
	\begin{equation}
	\alpha_j^\top\bar{\boldX}-\beta_j+\|(\bm{\mathcal{A}}\mathbf{\Sigma}_0\bm{\mathcal{A}}^\top+\bm{\mathcal{D}}\bm{\mathcal{D}}^\top)^{1/2}(I+\bm{\mathcal{B}}K)^\top\alpha_j\| \, \Phi^{-1}\left(1-p_j\right)\leq0,
	\end{equation}
	where $\|\cdot\|$ denotes the 2-norm of a vector. 
	Using Eq.~(\ref{eq:barX}), inequality (\ref{eq:deterministicChanceCosntraintBilinear}) follows.
\end{proof}

Note that the inequality constraint (\ref{eq:deterministicChanceCosntraintBilinear}) is a bilinear constraint, which makes it difficult to efficiently solve this problem. 
Thus, we convert the  chance constraints~(\ref{eq:conservativeChanceConstraints}) as follows
\begin{subequations}
	\begin{align}
	\texttt{Pr}(\alpha_j^\top \boldX > \beta_j) &\leq p_{j,\rm{fail}},\qquad  j = 1,\ldots M,\\
	\sum_{j=1}^{M}p_{j,\rm{fail}} &\leq P_{\rm{fail}}. \label{eq:sumofPfailure}
	\end{align}\label{eq:prespecifiedconservativeChanceConstraints}
\end{subequations}
Note that, unlike $p_j$, the $p_{j,\rm{fail}}$ is not a decision variable but a pre-specified value satisfying inequality~(\ref{eq:sumofPfailure}).
This alternative formulation implies that one needs to specify $p_j$ a priori. 

In summary, the chance constraints are formulated as follows. 
\begin{equation}
\alpha_j^\top(I+\bm{\mathcal{B}}K)\bm{\mathcal{A}}\boldsymbol{\mu}_0 + \|(\bm{\mathcal{A}}\mathbf{\Sigma}_0\bm{\mathcal{A}}^\top+\bm{\mathcal{D}}\bm{\mathcal{D}}^\top)^{1/2}(I+\bm{\mathcal{B}}K)^\top\alpha_j\| \Phi^{-1}\left(1-p_{j,\rm{fail}}\right) - \beta_j \leq 0. \label{eq:deterministicChanceConstraintsOurs}
\end{equation}
Note that, unlike the case described in Section \ref{sec:NoChanceConstraint}, where no chance constraints exist, we cannot decouple the mean and covariance steering problems owing to (\ref{eq:deterministicChanceConstraintsOurs}).

\subsection{Terminal Gaussian Distribution Constraint}
When solving covariance steering problems, we steer the mean to the pre-specified $\mu_N$ as well as the covariance to the pre-specified $\Sigma_N$. 
While we are converting the original chance-constrained covariance steering problem to a convex programming problem, as discussed in Section~\ref{sec:CSOnly}, the terminal covariance constraint~(\ref{eq:SigmaN}) is not convex.
We therefore relax this constraint to the following inequality constraint
\begin{equation}   \label{eq:terminalineq}
	\mathbb{E}[\tilde{x}_N\tilde{x}_N^\top] \preceq \Sigma_N.
\end{equation}
This condition implies that the covariance of the terminal state is smaller than a pre-specified $\Sigma_N$ which is reasonable in practice.
Note also that this change of terminal constraint relaxes the chance-constraint requirement for $\Sigma_N$ as well. 
Namely, if $\mu_N$ is inside the feasible region, $\Sigma_N$ can be any value as far as it is positive definite. 
We are now ready to prove the following result.

\begin{Proposition}

	The terminal constraints (\ref{eq:muN}) and (\ref{eq:terminalineq}) can be formulated as
	\begin{align}
		&\mu_N = \mathbf{E}_N(I+\bm{\mathcal{B}}K)\bm{\mathcal{A}}\boldsymbol{\mu}_0, \label{eq:terminalConstraintmu}\\
		&1 - \| (\bm{\mathcal{A}}\mathbf{\Sigma}_0\bm{\mathcal{A}}^\top+\bm{\mathcal{D}}\bm{\mathcal{D}}^\top)^{1/2}(I+\bm{\mathcal{B}}K)^\top \mathbf{E}_N^\top\Sigma_N^{-1/2}  \|_2 \geq 0, \label{eq:finalBoundaryCondition}
	\end{align}
	where $\mathbf{E}_N \triangleq \begin{bmatrix}0_{n_x} & E_N	\end{bmatrix} \in \Re^{n_x\times (N+2)n_x}$, and $\|\cdot\|_2$ denotes the 2 norm of a matrix.
\end{Proposition}

\begin{proof}
	It follows from~(\ref{eq:barX}) and (\ref{eq:SigmaX}) that 
	\begin{align}
		\mathbb{E}\left[x_N\right] &= \mathbf{E}_N(I+\bm{\mathcal{B}}K)\bm{\mathcal{A}}\boldsymbol{\mu}_0,    \label{eq:terminalConstraintmu2}\\
		\mathbb{E}\left[\tilde{x}_N\tilde{x}_N^\top\right] &= \mathbf{E}_N(I+\bm{\mathcal{B}}K)(\bm{\mathcal{A}}\mathbf{\Sigma}_0\bm{\mathcal{A}}^\top+\bm{\mathcal{D}}\bm{\mathcal{D}}^\top)(I+\bm{\mathcal{B}}K)^\top \mathbf{E}_N^\top.\label{eq:terminalconstraintEq}
	\end{align} 
Using inequality (\ref{eq:terminalineq}) it follows that~(\ref{eq:terminalconstraintEq}) results in the following inequality constraint, which is convex in $K$
	\begin{equation}
		\mathbf{E}_N(I+\bm{\mathcal{B}}K)(\bm{\mathcal{A}}\mathbf{\Sigma}_0\bm{\mathcal{A}}^\top+\bm{\mathcal{D}}\bm{\mathcal{D}}^\top)(I+\bm{\mathcal{B}}K)^\top \mathbf{E}_N^\top \preceq \Sigma_N.\label{eq:terminalConstraintSigma}
	\end{equation}
Since by assumption $\Sigma_N \succ 0$, inequality~(\ref{eq:terminalConstraintSigma}) becomes
	\begin{equation}
		I_{n_x} - \Sigma_N^{-1/2}\mathbf{E}_N(I+\bm{\mathcal{B}}K)(\bm{\mathcal{A}}\mathbf{\Sigma}_0\bm{\mathcal{A}}^\top+\bm{\mathcal{D}}\bm{\mathcal{D}}^\top)(I+\bm{\mathcal{B}}K)^\top \mathbf{E}_N^\top\Sigma_N^{-1/2} \succeq 0.
	\end{equation}
As being symmetric, the matrix $\Sigma_N^{-1/2}\mathbf{E}_N(I+\bm{\mathcal{B}}K)(\bm{\mathcal{A}}\mathbf{\Sigma}_0\bm{\mathcal{A}}^\top+\bm{\mathcal{D}}\bm{\mathcal{D}}^\top)(I+\bm{\mathcal{B}}K)^\top \mathbf{E}_N^\top\Sigma_N^{-1/2}$ is diagonalizable via an orthogonal matrix $S \in \Re^{n_x \times n_x}$.
	Thus, 
	\begin{equation}
	S\left(I_{n_x} - \mathtt{diag}(\lambda_1,\ldots,\lambda_{n_x})\right)S^\top \succeq 0, \label{eq:SImdiagS}
	\end{equation}
	where $\lambda_1,\ldots,\lambda_{n_x}$ are the eigenvalues of  $\Sigma_N^{-1/2}\mathbf{E}_N(I+\bm{\mathcal{B}}K)(\bm{\mathcal{A}}\mathbf{\Sigma}_0\bm{\mathcal{A}}^\top+\bm{\mathcal{D}}\bm{\mathcal{D}}^\top)(I+\bm{\mathcal{B}}K)^\top \mathbf{E}_N^\top\Sigma_N^{-1/2}$.
	Inequality~(\ref{eq:SImdiagS}) is implied by
	\begin{equation}
		1 - \lambda_{\rm max}\left(\Sigma_N^{-1/2}\mathbf{E}_N(I+\bm{\mathcal{B}}K)(\bm{\mathcal{A}}\mathbf{\Sigma}_0\bm{\mathcal{A}}^\top+\bm{\mathcal{D}}\bm{\mathcal{D}}^\top)(I+\bm{\mathcal{B}}K)^\top \mathbf{E}_N^\top\Sigma_N^{-1/2} \right) \geq 0. \label{eq:1mlambdamatrix}
	\end{equation}
	Furthermore,
	\begin{subequations}
		\begin{align}
		&\lambda_{\rm max}\left(\Sigma_N^{-1/2}\mathbf{E}_N(I+\bm{\mathcal{B}}K)(\bm{\mathcal{A}}\mathbf{\Sigma}_0\bm{\mathcal{A}}^\top+\bm{\mathcal{D}}\bm{\mathcal{D}}^\top)(I+\bm{\mathcal{B}}K)^\top \mathbf{E}_N^\top\Sigma_N^{-1/2} \right), \\
		&=\lambda_{\rm max}\left(\Sigma_N^{-1/2}\mathbf{E}_N(I+\bm{\mathcal{B}}K)(\bm{\mathcal{A}}\mathbf{\Sigma}_0\bm{\mathcal{A}}^\top+\bm{\mathcal{D}}\bm{\mathcal{D}}^\top)^{1/2}(\bm{\mathcal{A}}\mathbf{\Sigma}_0\bm{\mathcal{A}}^\top+\bm{\mathcal{D}}\bm{\mathcal{D}}^\top)^{1/2}(I+\bm{\mathcal{B}}K)^\top \mathbf{E}_N^\top\Sigma_N^{-1/2} \right), \\
		&=\lambda_{\rm max}\left(\left((\bm{\mathcal{A}}\mathbf{\Sigma}_0\bm{\mathcal{A}}^\top+\bm{\mathcal{D}}\bm{\mathcal{D}}^\top)^{1/2}(I+\bm{\mathcal{B}}K)^\top \mathbf{E}_N^\top\Sigma_N^{-1/2}\right)^\top\left((\bm{\mathcal{A}}\mathbf{\Sigma}_0\bm{\mathcal{A}}^\top+\bm{\mathcal{D}}\bm{\mathcal{D}}^\top)^{1/2}(I+\bm{\mathcal{B}}K)^\top \mathbf{E}_N^\top\Sigma_N^{-1/2} \right) \right),\\
		&=\|(\bm{\mathcal{A}}\mathbf{\Sigma}_0\bm{\mathcal{A}}^\top+\bm{\mathcal{D}}\bm{\mathcal{D}}^\top)^{1/2}(I+\bm{\mathcal{B}}K)^\top \mathbf{E}_N^\top\Sigma_N^{-1/2}\|_2^2,
		\end{align}
	\end{subequations}
It follows that (\ref{eq:1mlambdamatrix}) is equivalent to 
	\begin{equation}  \label{eq:1mnorm}
	1 - \|(\bm{\mathcal{A}}\mathbf{\Sigma}_0\bm{\mathcal{A}}^\top+\bm{\mathcal{D}}\bm{\mathcal{D}}^\top)^{1/2}(I+\bm{\mathcal{B}}K)^\top \mathbf{E}_N^\top\Sigma_N^{-1/2}\|_2^2 \geq 0. 
	\end{equation}
\end{proof}

\begin{remark}
	The inequality constraint (\ref{eq:terminalConstraintSigma}) can also be implemented by  taking the Schur complement
	\begin{equation}
	\begin{bmatrix}
	\Sigma_N & \mathbf{E}_N(I+\bm{\mathcal{B}}K)(\bm{\mathcal{A}}\mathbf{\Sigma}_0\bm{\mathcal{A}}^\top+\bm{\mathcal{DD}}^\top)^{1/2}\\[5pt]
	(\bm{\mathcal{A}}\mathbf{\Sigma}_0\bm{\mathcal{A}}^\top+\bm{\mathcal{DD}}^\top)^{1/2}(I+\bm{\mathcal{B}}K)^\top \mathbf{E}_N^\top& I_{(N+2)n_x}
	\end{bmatrix}\succeq 0.
	\end{equation}
	This is the approach followed in~\cite{bakolas2016optimalCDC}.
\end{remark}

In summary, the problem this work solves is as follows:

\begin{Problem}\label{prob:ConvexProblem}
	Given the system state sequence Eqs.~(\ref{eq:barX}) and (\ref{eq:SigmaX}),	
	find the control gain $K$ in Eq.~(\ref{eq:Kdef}), the shape of which is specified as in Eq.~(\ref{eq:KK1KX}), that minimizes the objective function Eq.~(\ref{eq:Kobj2Solve}) subject to the terminal constraints Eqs.~(\ref{eq:terminalConstraintmu}) and (\ref{eq:finalBoundaryCondition}), and the chance constraints with pre-specified probability thresholds $p_{j,\rm{fail}}$ Eq.~(\ref{eq:deterministicChanceConstraintsOurs}).
\end{Problem}
Note that, because of the entries of $\bm{\mathcal{A}}$, $\bm{\mathcal{B}}$, and $\bm{\mathcal{D}}$, the initial state constraint Eq.~(\ref{eq:X0augmented}) is incorporated in the state sequence Eqs.~(\ref{eq:barX}) and (\ref{eq:SigmaX}).
It is also worth noticing that, unlike Problem \ref{prob:OriginalProblem}, Problem \ref{prob:ConvexProblem} is a convex programming problem, which is efficiently solvable using a nonlinear solver.

\section{Numerical Simulations}\label{sec:NumericalSimulation}

In this section we validate the proposed algorithm using a simple numerical example.
We use CVX \cite{cvx} with MOSEK \cite{mosek} to solve the relevant optimization problems. 

\subsection{Double Integrator} 
We consider the following linear time invariant system
\begin{equation}
	x_{k+1} = Ax_k + Bu_k + Dw_k,
\end{equation}
where $x_k \in \Re^{2}$, $u_k \in \Re$, and $w_k \in \Re^{2}$ and
\begin{equation}
	A = \begin{bmatrix}
	1 & 1 \\ 0 & 1
	\end{bmatrix},\qquad
	B = \begin{bmatrix}
	0 \\ 1
	\end{bmatrix},\qquad
	D = \begin{bmatrix}
	0.01 & 0 \\ 0 &0.01
	\end{bmatrix}.
\end{equation}
We also consider the following  objective function
\begin{equation}
J(u_0,\ldots,u_{N-1}) = \sum_{k=0}^{N-1}u_k^\top u_k,
\end{equation}
with the boundary conditions
\begin{align}
&\mu_0 = \begin{bmatrix}
0\\8
\end{bmatrix},
\qquad
\Sigma_0 = \begin{bmatrix}
1 & 0 \\ 0 & 0.5
\end{bmatrix},\\
&\mu_N = \begin{bmatrix}
6\\0
\end{bmatrix},
\qquad
\Sigma_N = \begin{bmatrix}
0.5 & 0 \\ 0 & 0.5
\end{bmatrix}.
\end{align}
In this numerical simulation, we set $N=10$.

Furthermore, we consider the following chance constraint: For all $k \in [0,\ldots,N]$
\begin{equation}
	\mathtt{Pr}(
	\begin{bmatrix}
	1 & 1
	\end{bmatrix}
	x_k \leq 20) > 1-p_{\rm{fail}}.
\end{equation}
Note that, if we do not steer the covariance, because the state covariance evolves according to
\begin{equation}
\Sigma_X = \mathcal{A}\Sigma_0\mathcal{A}^\top + \mathcal{DD}^\top,
\end{equation} 
the terminal state covariance $E_N\Sigma_XE_N^\top$  becomes large at the final step $N$, and this chance constraint cannot be satisfied.
The problem becomes infeasible, as illustrated in Fig.~\ref{fig:meanSteerOnly}.
The red line denotes the expected trajectory, computed using~(\ref{eq:meanSteerController}), and red ellipses denote the pre-specified $3\sigma$ bounds of the initial and terminal state distributions, and each blue ellipse represents the predicted $3\sigma$ bounds at each time step. 
Gray lines are the trajectories starting from 100 different initial conditions, which are sampled from $\mathcal{N}(\mu_0,\Sigma_0)$.
\begin{figure}[htb]
	\centering
	\includegraphics[width=0.6\columnwidth]{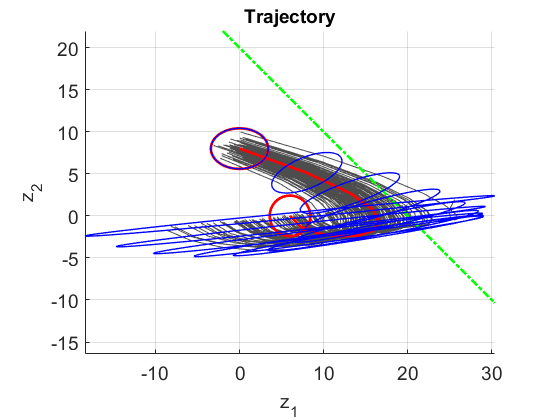}
	\caption{Result of mean steering.}\label{fig:meanSteerOnly}
\end{figure}
In this scenario, where we do not apply covariance steering and there are no chance constraints, the cost value was 23.90. 

Next, we steer the covariance so that the terminal covariance is less than the pre-specified covariance $\Sigma_N$ without the chance constraints. 
Figure~\ref{fig:CSNoConstraint} illustrates the results. 
The trajectories successfully converge to a region with covariance less than $\Sigma_N$.
In this scenario, the cost value was 24.05 and is slightly higher than the mean steering case. 
This increase in cost represents the trade-off to steer the covariance, in addition to the mean.

\begin{figure}[htb]
	\centering
	\includegraphics[width=0.6\columnwidth]{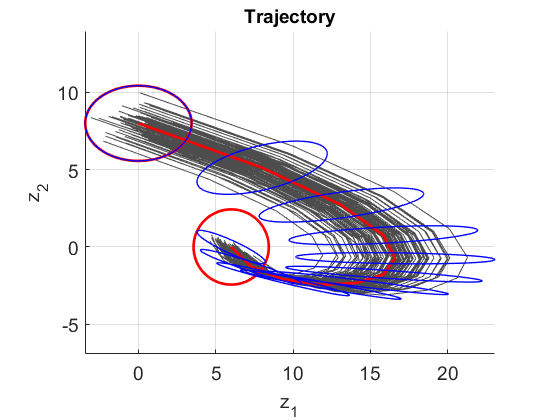}
	\caption{Result of covariance steering without chance constraints.}\label{fig:CSNoConstraint}
\end{figure}

Finally, we consider the case with chance constraints, which is shown in Fig.~\ref{fig:CSChanceConstraint}.
We set the probability threshold of failure to be $p_{j,\rm{fail}}=0.001$.
The blue ellipses in Fig.~\ref{fig:CSChanceConstraint} represent the corresponding confidence regions of the state, and some of them touch the constraints.  
In this scenario, the cost increased to 24.16. 

\begin{figure}[htb]
	\centering
	\includegraphics[width=0.6\columnwidth]{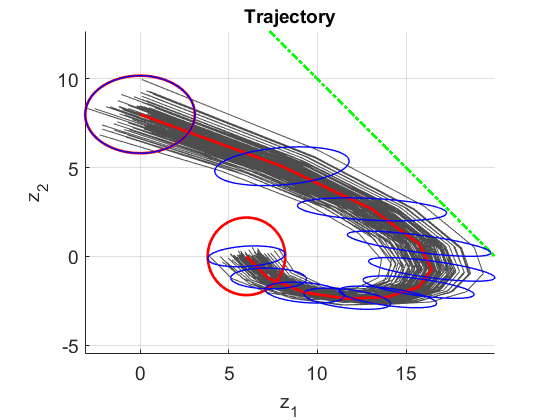}
	\caption{Result of covariance steering with chance constraints.}\label{fig:CSChanceConstraint}
\end{figure}

\subsection{Vehicle Path Planning}

Next, we consider the path-planning problem for a vehicle under the following time invariant system dynamics with $x_k = [x,y,v_x,v_y]^\top\in \Re^{4}$, $u_k =[a_x,a_y]^\top\in \Re^{2}$, $w_k \in \Re^{4}$ and
\begin{equation}
A = \begin{bmatrix}
1 & 0 & \Delta t & 0\\
0 & 1 & 0 & \Delta t\\
0 & 0 & 1 & 0\\
0 & 0 & 0 & 1
\end{bmatrix},\qquad
B = \begin{bmatrix}
\Delta t^2 & 0\\
0 &\Delta t^2\\
\Delta t & 0\\
0 & \Delta t
\end{bmatrix},\qquad
D = \begin{bmatrix}
0.01 & 0 & 0 & 0 \\
 0 & 0.01 & 0 & 0\\
 0 & 0 & 0.01 & 0\\
 0 & 0 & 0 & 0.01
\end{bmatrix},
\end{equation}
where $\Delta t$ is the time-step size, and we set $\Delta t = 0.2$. 
Figure \ref{fig:2ProblemSetup} illustrates the problem setup. 
The red circle denotes the $3\sigma$ error of the initial state distribution of $x$ and $y$ coordinates.
The magenta circle denotes the $3\sigma$ error of the terminal state distribution of $x$ and $y$ coordinates.
Specifically, the initial condition is 
\begin{equation}
\mu_0 = [-10, 1, 0, 0],\qquad \Sigma_0 = \mathtt{diag}(0.1,0.1,0.01,0.01),
\end{equation}
while the terminal constraint is 
\begin{equation}
\mu_N = [0, 0, 0, 0],\qquad \Sigma_N = 0.5\Sigma_0.
\end{equation}

The green dotted lines illustrate the state constraints. 
Specifically, 
\begin{equation}
\frac{1}{5}(x-1) \leq y \leq -\frac{1}{5}(x-1)
\end{equation}
The vehicle has to remain in the region between the two lines while moving from the red to the magenta regions. 
Such a ``cone''-shaped constraint is seen in many engineering applications, e.g., the instrument landing for aircraft, spacecraft rendezvous, and drone-landing on a moving platform. 
The probabilistic threshold for the violation of chance constraints was specified a priori, and we set $p_{j,\rm{fail}} = 0.0005$.
\begin{figure}[htb]
	\centering
	\includegraphics[width=0.6\columnwidth]{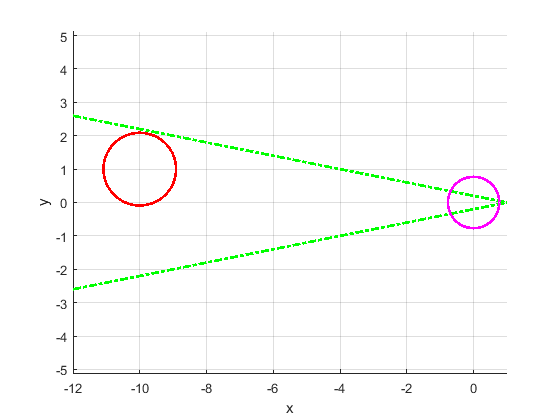}
	\caption{Problem setup for the second numerical example.}\label{fig:2ProblemSetup}
\end{figure}
The objective function is 
\begin{equation}
J(x_0,\ldots,x_{N-1},u_0,\ldots,u_{N-1}) = \sum_{k=0}^{N-1}x_k^\top Q_k x_k + u_k^\top R_k u_k,
\end{equation}
where  
\begin{equation}
Q_k = \mathtt{diag}(10,10,1,1),\qquad
R_k = \mathtt{diag}(10^3,10^3),
\end{equation}
with horizon $N=20$.
This problem is infeasible if we do not control the state covariance. 
See, for example, Figure~\ref{fig:2meanSteer}, which shows the results using only the mean steering controller~(\ref{eq:meanSteerController}).
As the covariance grows, it is impossible to find a feasible solution to this problem that will guarantee the satisfaction of chance constraints. 

\begin{figure}[htb]
	\centering
	\includegraphics[width=0.6\columnwidth]{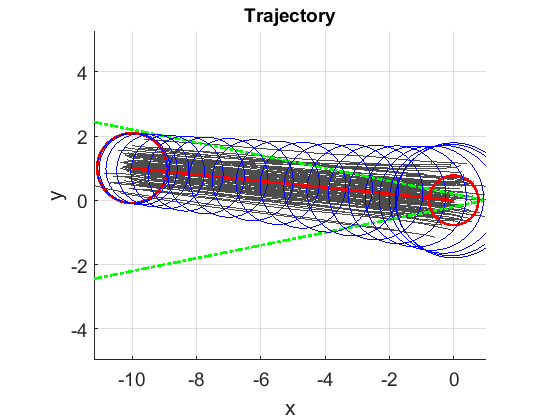}
	\caption{Mean steering results.}\label{fig:2meanSteer}
\end{figure}

Before discussing the case with chance constraints, we discuss the case without chance constraints, which is illustrated in Fig.~\ref{fig:2CSNC}.
By introducing the covariance steering, the uncertainty of the future trajectory successfully reduced.  

\begin{figure}[htb]
	\centering
	\includegraphics[width=0.6\columnwidth]{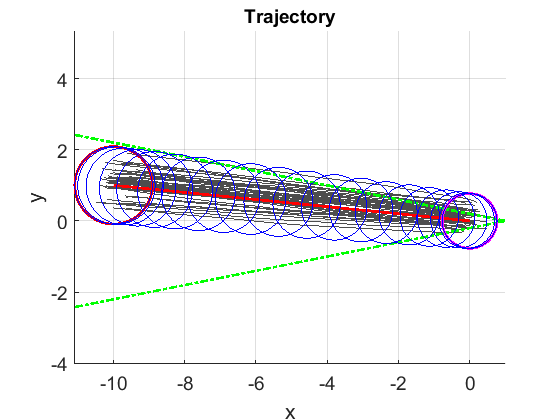}
	\caption{Covariance steering results without chance constraints.}\label{fig:2CSNC}
\end{figure}

Finally, Fig.~\ref{fig:2CCCS} illustrates the results of the proposed chance-constrained covariance steering approach. 
The error ellipse successfully changed its shape to avoid collision with the constraints while maintaining the terminal covariance constraints to be less than the pre-specified state covariance.

\begin{figure}[htb]
	\centering
	\includegraphics[width=0.6\columnwidth]{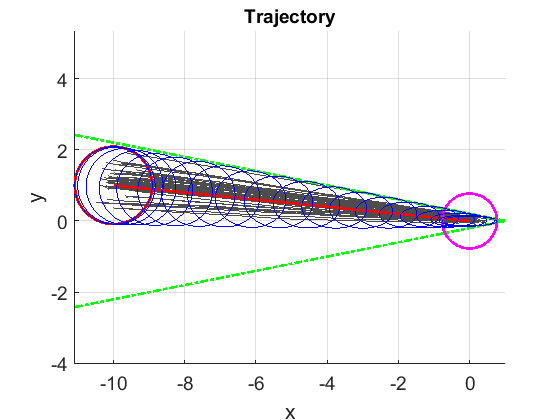}
	\caption{Covariance steering results with chance constraints.}\label{fig:2CCCS}
\end{figure}

\section{Summary}\label{sec:Summary}
This work has addressed the problem of optimal steering of the covariance for a stochastic linear time-varying system subject to chance constraints in discrete time.
We showed that if there are no chance constraints, we can independently design the mean and covariance steering controllers, and we introduced an analytic solution for the mean steering. 
We also showed that the optimal covariance steering problem with chance constraints can be converted to a convex programming problem.
The proposed approach was verified using numerical examples.
Future work will investigate the applications of the proposed approach to stochastic model predictive controllers.

\bibliographystyle{IEEEtran}
\bibliography{ChanceConstrained_CS}

\end{document}